\documentclass[11pt]{amsart}

\usepackage[a4paper,hmargin=3.5cm,vmargin=4cm]{geometry}
\usepackage{amsfonts,amssymb,amscd,amstext}
\usepackage{graphicx}
\usepackage[dvips]{epsfig}
\usepackage{hyperref}

\usepackage{fancyhdr}
\input xy
\xyoption{all}

\usepackage{enumerate}
\usepackage{titlesec}
\usepackage{mathrsfs}

\titleformat{\section}%[display]
{\filcenter\bfseries\large} {\thesection{.}}{0.2cm}{}%[$\vspace*{-1.0cm}$]
\titleformat{\subsection}[runin]
{\bfseries} {\thesubsection{.}}{0.15cm}{}[.]
\titleformat{\subsubsection}[runin]
{\em}{\thesubsubsection{.}}{0.15cm}{}[.]
\usepackage[up,bf]{caption}
\newtheorem{theorem}{Theorem}[section]

\newtheorem{claim}[theorem]{Claim}
\newtheorem{lemma}[theorem]{Lemma}
\newtheorem{corollary}[theorem]{Corollary}

\theoremstyle{definition}

\numberwithin{equation}{section}
\numberwithin{figure}{section}

\newcommand\Ocal{\mathcal{O}}

\newcommand\Ascr{\mathscr{A}}

\newcommand\Cscr{\mathscr{C}}

\newcommand\Lscr{\mathscr{L}}

\newcommand\C{\mathbb{C}}

\renewcommand\b{\mathbb{B}}
\renewcommand\c{\mathbb{C}}

\renewcommand\d{\mathbb D}

\newcommand\n{\mathbb{N}}
\renewcommand\r{\mathbb{R}}

\newcommand\z{\mathbb{Z}}

\newcommand\hra{\hookrightarrow}

\newcommand\wt{\widetilde}
\newcommand\wh{\widehat}

\newcommand\dist{\mathrm{dist}}

\def\dist{\mathrm{dist}}

\usepackage{color}

\begin{document}

%%
%% Headings
%%

\fancyhead[LO]{Complex curves in pseudoconvex Runge domains}
\fancyhead[RE]{A.\ Alarc\'on}
\fancyhead[RO,LE]{\thepage}

\thispagestyle{empty}

%%
%% Title
%%

\vspace*{1cm}
\begin{center}
{\bf\LARGE Complex curves in pseudoconvex Runge domains containing discrete subsets}

\vspace*{0.5cm}

%
% Authors
%
{\large\bf Antonio Alarc\'on}
\end{center}

%% Addresses and finantial support
%\footnote[0]{\vspace*{-0.4cm}
%}

%
% Abstract, keywords, and MSC
%
\vspace*{5mm}

\begin{quote}
{\small
\noindent {\bf Abstract}\hspace*{0.1cm}
For any pseudoconvex Runge domain $\Omega\subset\c^2$ we prove that every closed discrete subset in $\Omega$ is contained in a properly embedded complex curve in $\Omega$ with any prescribed topology (possibly infinite). 

\vspace*{1mm}

\noindent{\bf Keywords}\hspace*{0.1cm} 
complex curve, 
holomorphic embedding, 
pseudoconvex domain, 
Runge domain.

\vspace*{1mm}

\noindent{\bf Mathematics Subject Classification (2010)}\hspace*{0.1cm} 
32H02, % Holomorphic mappings, (holomorphic) embeddings and related questions - SCV and Analytic spaces
32E30, % Holomorphic and polynomial approximation, Runge pairs, interpolation - SCV  
32T05. % Domains of holomorphy - SCV
}
\end{quote}

%\vspace*{5mm}

%%%%%%%%%%
%%%%%%%%%%
%%%%%%%%%% Section: Introduction
%%%%%%%%%%
%%%%%%%%%%
%%%%%%%%%%

\section{Introduction} 
\label{sec:intro}

A problem that has been the focus of interest is to determine whether a domain $\Omega$ in a complex Euclidean space $\c^N$ $(N\ge 2)$ admits closed complex curves containing a given closed discrete subset of $\Omega$ (see, among others, \cite{Globevnik1988MA,Globevnik1989IM,ForstnericGlobevnik1992CMH,RosayRudin1993MN,ForstnericGlobevnikRosay1996AM}). In this paper we are interested in the particular case of embedded curves in domains of $\c^2$. To the best of the author's knowledge, the most general result so far in this direction dates back to 1996 and says that for any {\em pseudoconvex Runge domain} $\Omega\subset\c^2$ and any closed discrete subset $\Lambda\subset\Omega$ there is a proper holomorphic embedding from the open unit disc $\d\subset\c$ to $\Omega$ whose image contains $\Lambda$ (see Forstneri\v c, Globevnik, and Stens\o nes \cite{ForstnericGlobevnikStensones1996MA}). Such disc was found as a leaf in a holomorphic foliation on $\Omega$ by holomorphic discs. It is moreover very likely that a slight refinement of the construction in \cite{ForstnericGlobevnikStensones1996MA} provides, for any such $\Omega$ and $\Lambda$, properly embedded complex curves in $\Omega$, containing $\Lambda$, with any {\em finite} topology; a trickier question is whether there are such curves with arbitrary (possibly infinite) topology. This paper gives an affirmative answer;
here is a simplified version of our main result (Theorem \ref{th:main}).
%%%
%%% Theorem
%%%
\begin{theorem}\label{th:main-intro}
Let $\Omega\subset\c^2$ be a pseudoconvex Runge domain, $M$ be an open Riemann surface, 
and $E\subset M$ be a closed discrete subset. For any proper injective map $f\colon E\to\Omega$ there is a Runge domain $D\subset M$ such that $E\subset D$, $D$ is a deformation retract of $M$, and the map $f$ extends to $D$ as a proper holomorphic embedding into $\Omega$.
\end{theorem}

If we choose $M=\d$, then the domain $D$ furnished by Theorem \ref{th:main-intro} is Runge in $\c$ and relatively compact, hence biholomorphic to the unit disc $\d$; we thereby recover the above mentioned result from \cite{ForstnericGlobevnikStensones1996MA}. However, in general, one cannot choose $D$ to be biholomorphic to $M$ in the theorem; for instance, whenever that $\Omega$ is bounded and $M$ is {\em Liouville} (also called {\em parabolic}, i.e., carrying no negative non-constant subharmonic functions). On the other hand, we shall prove that the domain $D\subset M$ can always be chosen of {\em hyperbolic} type (i.e., carrying negative non-constant subharmonic functions).

The subset $f(E)\subset\Omega$ given in Theorem \ref{th:main-intro} is closed and discrete; reciprocally, every closed discrete subset $\Lambda\subset \Omega$ is of the form $\Lambda=f(E)$ for some $E$ and $f$ as in the theorem. Moreover, the furnished domain $D$ is homeomorphic, hence also diffeomorphic, to the arbitrarily given open Riemann surface $M$. We therefore obtain the following corollary.
%
% Corollary
%
\begin{corollary}\label{co:intro}
Let $\Omega\subset\c^2$ be a pseudoconvex Runge domain and $\Lambda\subset \Omega$ be a closed discrete subset. On each open connected orientable smooth surface $M$ there is a complex structure $J$ such that the open Riemann surface $R=(M,J)$ admits a proper holomorphic embedding into $\Omega$ whose image contains $\Lambda$.
\end{corollary}

Theorem \ref{th:main-intro} is already known in the particular cases when the domain $\Omega$ is $\c^2$ (see Ritter \cite{Ritter2014Crelle}) and when it is the open unit ball in $\c^2$ (see Globevnik and the author \cite{AlarconGlobevnik2017C2}). 
We point out that the assumptions on $\Omega$ (i.e., pseudoconvexity and having the Runge property) cannot be entirely removed from the statement of the theorem. Indeed, there are smoothly bounded relatively compact domains $\Omega\subset\c^2$ for which there is no proper holomorphic map from $\d$ to $\Omega$ passing through a certain point $z\in\Omega$ (see Forstneri\v c and Globevnik \cite{ForstnericGlobevnik1992CMH}); and also bounded domains in $\c^2$ containing no proper images of $\d$ (see Dor \cite{Dor1996MZ}).
It is an open question whether Theorem \ref{th:main-intro} remains valid when $\Omega$ is an arbitrary (non-Runge) pseudoconvex domain in $\c^2$, even in case $M=\d$ (see \cite[p.\ 559]{ForstnericGlobevnikStensones1996MA}). To this respect, it is known that the conclusion of Theorem \ref{th:main-intro} holds true for the pseudoconvex domains $\Omega=\c\times(\c\setminus\{0\})$ and $\Omega=(\c\setminus\{0\})^2$ (see Ritter \cite{Ritter2014Crelle} and also L\'arusson and Ritter \cite{LarussonRitter2014IUMJ}), which are not Runge in $\c^2$. 

The main achievement of Theorem \ref{th:main-intro} is of course the embeddedness of the examples. Recall that self-intersections of complex curves in $\c^2$ are stable under small deformations; this is why, in general, constructing embedded complex curves in $\c^2$ is a much more demanding task that in $\C^N$ for $N\ge 3$. Regarding this, we point out that the main result in \cite{ForstnericGlobevnikStensones1996MA}, which we recalled at the beginning of this introduction, is actually established for pseudoconvex Runge domains in $\c^N$ for arbitrary dimension $N\ge 2$. Furthermore, the analogue of Theorem \ref{th:main-intro} for arbitrary (possibly non-Runge) pseudoconvex domains in $\c^N$ for $N\ge 3$ easily follows from the results by Forstneri\v c and Slapar in \cite{ForstnericSlapar2007MRL} (see also \cite[\textsection 9.10]{Forstneric2017}) and by Drinovec Drnov\v sek and Forstneri\v c in \cite{DrinovecForstneric2007DMJ,DrinovecForstneric2010AJM}, even choosing the domain $D$ to agree with $M$ provided that $M$ is a bordered Riemann surface. Elaborating from the same results also shows that the conclusion of Theorem \ref{th:main-intro} holds true for pseudoconvex domains in $\c^2$ if one allows the holomorphic curves to have self-intersections. Previous partial results in this direction can be found in Globevnik \cite{Globevnik1988MA,Globevnik1989IM} and Forstneri\v c and Globevnik \cite{ForstnericGlobevnik1992CMH}. 
Looking at more general targets than pseudoconvex domains in $\c^N$, the analogue of Theorem \ref{th:main-intro} holds true for holomorphic embeddings into any Stein manifold, of dimension at least three, having the density property; in fact, in this framework there is no need to shrink the initial open Riemann surface and one may choose $D=M$ (see Andrist and Wold \cite{AndristWold2014AIF} and Andrist, Forstneri\v c, Ritter, and Wold \cite{AndristForstnericRitterWold2016JAM}). The same is valid for holomorphic {\em immersions} into any Stein surface with the density property (see \cite{AndristWold2014AIF,AndristForstnericRitterWold2016JAM} and also Forstneri\v c \cite{Forstneric2017Immersions}). Even more generally, in light of the results in \cite{DrinovecForstneric2007DMJ} one is led to expect that the analogue of Theorem \ref{th:main-intro} should also hold true for holomorphic immersions into an arbitrary Stein surface and for holomorphic embeddings into an arbitrary Stein manifold of dimension greater than two (without asking them to enjoy the density property).

Our proof follows the usual strategy of adding either a handle or an end at each step in a recursive process used in the construction of complex curves with arbitrary topology (see \cite{AlarconLopez2013JGA}), but with nontrivial modifications which enable to ensure the hitting condition and the properness in the given pseudoconvex Runge domain.
Besides some of the nice properties of these domains (see Section \ref{sec:prelim}), our method relies on the classical Mergelyan approximation theorem for holomorphic functions 
and the theory of holomorphic automorphisms of complex Euclidean spaces. The latter has already shown to be a powerful tool for constructing embedded complex submanifolds in $\c^N$ for $N\ge 2$, in particular, holomorphic curves in $\c^2$ (we refer to \cite[Chapter 4]{Forstneric2017} for a survey of results in the subject). In particular, the use of holomorphic automorphisms was crucial  in the method, different from ours, developed by Forstneri\v c, Globevnik, and Stens\o nes in \cite{ForstnericGlobevnikStensones1996MA}.

\subsubsection*{Outline of the paper}
Section \ref{sec:prelim} is devoted to introduce some notation and recall the basic concepts, definitions, and results that will be needed throughout this paper. In Section \ref{sec:main-intro} we state the main result of the paper (Theorem \ref{th:main}) and show how it implies Theorem \ref{th:main-intro}. Finally, we prove Theorem \ref{th:main} in Section \ref{sec:MT}.

%%%%%%%%%%
%%%%%%%%%%
%%%%%%%%%%
%%%%%%%%%%   Section: Preliminares
%%%%%%%%%%
%%%%%%%%%%

\section{Preliminaries}\label{sec:prelim}

Given subsets $A$ and $B$ of a topological space $X$ we shall use the notation $A\Subset B$ to mean that the closure $\overline A$ of $A$ is contained in the interior $\mathring B$ of $B$. We also denote by $bA=\overline A\setminus\mathring A$ the frontier of $A$ in $X$. A subset $E\subset X$ is said to be {\em discrete} if every point in $E$ is isolated; if $E$ is closed, then this is equivalent to that no point of $X$ is a limit point of $E$. By a {\em domain} in $X$ we mean an open connected set, and the closure of a relatively compact domain shall be said to be a {\em compact domain}.

We denote $\n=\{1,2,3,\ldots\}$ and $\z_+=\n\cup\{0\}$. For any $N\in\n$ we denote by $|\cdot|$ and $\dist(\cdot,\cdot)$ the Euclidean norm and distance in the $\c^N$, respectively. 

Let $N\in\n$. A domain $\Omega\subset\c^N$ is said to be {\em pseudoconvex} if it carries a strictly plurisubharmonic exhaustion function; this happens if and only if $\Omega$ is holomorphically convex, if and only if $\Omega$ is a domain of holomorphy, and if and only if $\Omega$ is Stein. 
%Every domain in $\c$ is pseudoconvex; this is why pseudoconvexity is usually defined only for domains in higher dimensions.
%Every convex domain in $\c^N$ is pseudoconvex but the converse is not true. 
See Range \cite{Range2012NAMS} for a brief introduction to pseudoconvexity, and e.\ g.\  H\"ormander \cite{Hormander1990NH} or Range \cite{Range1986GTM} for further developments.

A domain $\Omega\subset\c^N$ is said to be {\em Runge} if every holomorphic function on $\Omega$ may be approximated, uniformly on compact subsets in $\Omega$, by holomorphic polynomials on $\c^N$. Likewise, a compact subset $L\subset \c^N$ is said to be {\em polynomially convex} if for each point $z \in\c^N\setminus L$ there is a (holomorphic) polynomial $P$ such that $|P(z)|>\sup\{|P(w)|\colon w\in L\}$; equivalently, if every holomorphic function on a neighborhood of $L$ may be approximated, uniformly on $L$, by polynomials on $\c^N$. We refer to Stout \cite{Stout2007PM} for a monograph on polynomial convexity. 

Given a pseudoconvex Runge domain $\Omega\subset\c^N$, a smooth plurisubharmonic exhaustion function $\varrho$ for $\Omega$, and a number $c\in\r$, the set
\[
	\Omega_c=\{z\in\Omega\colon \varrho(z)< c\}
\]
is relatively compact in $\Omega$ and a (possibly disconnected) Runge domain in $\Omega$ (the latter meaning that every holomorphic function on $\Omega_c$ may be approximated, uniformly on compact subsets in $\Omega_c$, by holomorphic functions on $\Omega$; see \cite[Theorem 1.3.7]{Stout2007PM}). Moreover, the set
\[
	\{z\in\Omega\colon \varrho(z)\le c\}
\]
is a polynomially convex compact set in $\c^N$ (see \cite[p. 25-26]{Stout2007PM}).

A compact set $K$ in an open Riemann surface $M$ is said {\em $\Ocal(M)$-convex} (also called {\em holomorphically convex} or {\em Runge} in $M$) if every continuous function from $K$ to $\c$ being holomorphic on $\mathring K$ may be approximated, uniformly on $K$, by holomorphic functions on $M$. By the Runge-Mergelyan theorem (see \cite{Runge1885AM,Mergelyan1951DAN,Bishop1958PJM}), this happens if and only if $M\setminus K$ has no relatively compact connected components in $M$.

A connected complex manifold is said to be {\em Liouville} if it does not carry non-constant negative plurisubharmonic functions; open Riemann surfaces which are not Liouville (i.e., carrying negative non-constant subharmonic functions) are called {\em hyperbolic} (see Farkas and Kra \cite[p.\ 179]{FarkasKra1992Springer}). If a connected open Riemann surface is hyperbolic, then so is every connected domain on it (viewed as an open Riemann surface). Throughout the paper, we shall always assume that Riemann surfaces are connected unless the contrary is stated. 

A {\em compact bordered Riemann surface} is a compact Riemann surface $R$ with nonempty boundary $bR\subset R$ consisting of finitely many pairwise disjoint  smooth Jordan curves; the interior $\mathring R=R\setminus bR$ of $R$ is said to be a {\em bordered Riemann surface}. Every bordered Riemann surface is hyperbolic. It is classical that every compact bordered Riemann surface is diffeomorphic to a smoothly bounded compact domain in an open Riemann surface. We shall denote by $\Ascr^1(R)$ the space of functions from $R$ to $\c$ of class $\Cscr^1$ which are holomorphic on $\mathring R$. 

%%%%%%%%%%
%%%%%%%%%%
%%%%%%%%%%
%%%%%%%%%%  Section \ref{sec:main-intro}
%%%%%%%%%%
%%%%%%%%%%

\section{Statement of the main result and proof of Theorem \ref{th:main-intro}}\label{sec:main-intro}

The main result of the present paper may be stated as follows.

\begin{theorem}\label{th:main}
Let $\Omega\subset\c^2$ be a pseudoconvex Runge domain, $M$ be an open Riemann surface, $K\subset M$ be a connected, smoothly bounded, $\Ocal(M)$-convex compact domain, $E\subset M$ be a closed discrete subset, and 
\[
	f\colon K\cup E\to\Omega
\]
be a proper injective map such that $f|_K$ is a holomorphic embedding.
Then, given a number $\epsilon>0$ and a connected polynomially convex compact set $L\subset \Omega$ satisfying
\begin{equation}\label{eq:LhE}
	L\cap f(bK)=\varnothing
	\quad\text{and}\quad
	L\cap f(E\setminus K)=\varnothing,
\end{equation}
there are a Runge domain $D\subset M$ and a proper holomorphic embedding $\wt f\colon D\hra\Omega$ enjoying the following conditions:
\begin{itemize}
\item[\rm (a)] $K\cup E\subset D$ and the domain $D$ is a deformation retract of (and hence homeomorphic to) $M$.
\smallskip
\item[\rm (b)] $|\wt f(p)-f(p)|<\epsilon$ for all $p\in K$.
\smallskip
\item[\rm (c)] $\wt f(p)=f(p)$ for all $p\in E$.
\smallskip
\item[\rm (d)] $\wt f(D\setminus \mathring K)\cap L=\varnothing$.
\end{itemize}
Furthermore, the domain $D$ may be chosen of hyperbolic type.
\end{theorem}

Note that, since $E\subset M$ is closed and discrete, $K\subset M$ is compact, and $f|_K$ is continuous, the map $f$ is continuous as well. Thus, by compactness of $K$, the assumption that $f\colon K\cup E\to\Omega$ is proper is equivalent to that $f|_E\colon E\to\Omega$ is a proper map; i.e., $(f|_E)^{-1}(C)\subset E$ is finite for any compact set $C\subset\Omega$.

We defer the proof of Theorem \ref{th:main} to Section \ref{sec:MT}. Let us first see that it implies Theorem \ref{th:main-intro}.
%
% Proof of Theorem \ref{th:main-intro}
%
\begin{proof}[Proof of Theorem \ref{th:main-intro} assuming Theorem \ref{th:main}]
Let $\Omega\subset\c^2$ be a pseudoconvex Runge domain, $M$ be an open Riemann surface, $E\subset M$ be a closed discrete subset, and $f\colon E\to\Omega$ be a proper injective map. Choose a simply-connected, smoothly bounded, compact domain $K$ in $M$ with $K\cap E=\varnothing$, and extend $f$ to $K\cup E$ as an injective map that is a holomorphic embedding on $K$. Also choose a connected polynomially convex compact set $L\subset\Omega\setminus f(K\cup E)$. Theorem \ref{th:main} applied to these objects and any number $\epsilon>0$ furnishes a Runge domain $D\subset M$ and a proper holomorphic embedding $\wt f\colon D\hra \Omega$ such that $D$ contains $E$, $D$ is a deformation retract of $M$, and $\wt f(p)=f(p)$ for all points $p\in E$. This completes the proof of Theorem \ref{th:main-intro} under the assumption that Theorem \ref{th:main} is valid.
\end{proof}
Furthermore, since the domain $D\subset M$ in Theorem \ref{th:main} may be chosen of hyperbolic type, the argument in the above proof shows that the same holds true (as we claimed in the introduction) for the domain $D$ in Theorem \ref{th:main-intro}.

We finish this section with the following corollary of Theorem \ref{th:main}, which is a more precise version of Corollary \ref{co:intro}.
\begin{corollary}\label{co:main}
Let $\Omega$, $M$, $K$, $E$, and $f$ be as in Theorem \ref{th:main} and denote by $J_0$ the complex structure on $M$. Then there exists a complex structure $J$ on $M$ such that $J=J_0$ on a connected neighborhood of $K\cup E$ in $M$ and there is a map $\wt f\colon M\to\Omega$ that is a proper holomorphic embedding with respect to $J$, approximates $f$ uniformly on $K$, and $\wt f|_E=f$.
Furthermore, the complex structure $J$ may be chosen so that the open Riemann surface $(M,J)$ is hyperbolic.
\end{corollary}
\begin{proof}
Let $D\subset M$ and $\wt f\colon D\hra\Omega$ be the domain and the holomorphic embedding provided by Theorem \ref{th:main} applied to the given data. It is then clear that the complex structure $J$ on $M$ which makes it biholomorphic to $D$ and the map $\wt f\colon D=(M,J)\to\Omega$ satisfy the conclusion of the corollary. Note that the fact that $J=J_0$ on a connected neighborhood of $K\cup E$ is implied by condition {\rm (a)} in the theorem.
\end{proof}

%
% Proof of Theorem \ref{th:main}
%

\section{Proof of Theorem \ref{th:main}}\label{sec:MT}

Let $\Omega$, $M$, $K$, $E$, $f$, and $L$ be as in the statement of Theorem \ref{th:main}, and fix a positive number $\epsilon>0$.
Set $L_0:=L$ and choose an increasing sequence of (connected) polynomially convex compact domains
\begin{equation}\label{eq:cupLj}
	L_1\Subset L_2\Subset \cdots\Subset \bigcup_{j\in\n} L_j=\Omega
\end{equation}
in $\c^2$ such that $L_0\subset \mathring L_1$ and
\begin{equation}\label{eq:bLj}
	f(E)\cap bL_j=\varnothing \quad\text{for all $j\in\n$}.
\end{equation}
Such a sequence may be constructed as follows. Take a smooth plurisubharmonic exhaustion function $\varrho\colon\Omega\to\r$ (recall that $\Omega$ is pseudoconvex) and a sequence of real numbers $c_1<c_2<\cdots$ such that $\lim_{j\to\infty}c_j=+\infty$ and each $c_j$ is a regular value of $\varrho$ that satisfies $f(E)\cap \{z\in\Omega\colon \varrho(z)=c_j\}=\varnothing$ (recall that $f(E)\subset\Omega$ is closed and discrete).
Choose $c_1$ large enough so that $L_0\subset \{z\in\Omega\colon \varrho(z)<c_1\}$. Thus, it suffices to define $L_j$ as the connected component of $\{z\in\Omega\colon \varrho(z)\le c_j\}$ containing $L_0$, $j\in\n$.

Consider the following exhaustions of $f(E)\subset \Omega$ and $E\subset M$:
\begin{equation}\label{eq:Ej}
	\Lambda_j:=f(E)\cap L_j=f(E)\cap\mathring L_j \quad\text{and}\quad E_j:=f^{-1}(\Lambda_j)\subset,\quad j\in\n
\end{equation}
(see \eqref{eq:bLj} and recall that $f\colon K\cup E\to\Omega$ is injective.)
Thus, \eqref{eq:cupLj} ensures that $\Lambda_j\subset \Lambda_{j+1}$ for all $j\in\n$ and
\begin{equation}\label{eq:cupEj}
	E_1\subset E_2\subset\cdots\subset\bigcup_{j\in\n} E_j=E.
\end{equation}
Set $E_0:=E\cap K$ and $\Lambda_0:=f(E_0)$ and assume without loss of generality that $L_1$ is chosen large enough so that $f(K)\subset\mathring L_1$. We then have that $\Lambda_0\subset\Lambda_1$ and $E_0\subset E_1$. It is clear that if $E_j\neq\varnothing$ for a given $j\in\z_+$, then $E_j$ is finite and $f|_{E_j}\colon E_j\to\Lambda_j$ is a bijection; recall that the set $E\subset M$ is closed and discrete and the map $f|_E\colon E\to\Omega$ is proper and injective.
In the open Riemann surface $M$ we choose an increasing sequence of (connected) smoothly bounded $\Ocal(M)$-convex compact domains
\begin{equation}\label{eq:cupKj}
	K_0:=K\Subset K_1\Subset K_2\Subset \cdots\Subset \bigcup_{j\in\n} K_j=M
\end{equation}
with the property that the Euler characteristic
\begin{equation}\label{eq:Euler}
	\chi(K_j\setminus\mathring K_{j-1})\in\{-1,0\}\quad \text{for all $j\in\n$.}
\end{equation}
Such can be constructed by standard topological arguments; we refer for instance to \cite[Lemma 4.2]{AlarconLopez2013JGA} for a detailed proof. 

Finally, set $M_0:=K_0$ and $f_0:=f|_{M_0}\colon M_0\to\Omega$ and fix a number $0<\epsilon_0<\epsilon/2$. 

The main step in the proof of the theorem is enclosed in the following result.
%
% Lemma
%
\begin{lemma}\label{lem:Induction}
There are
\begin{itemize}
\item[\rm (A)] an increasing sequence of (connected) smoothly bounded $\Ocal(M)$-convex compact domains $M_1\Subset M_2\Subset \cdots$ with $M_0\subset \mathring M_1$,
\smallskip
\item[\rm (B)] a sequence of holomorphic embeddings $f_j\colon M_j\to \Omega$ $(j\in\n)$,  and
\smallskip
\item[\rm (C)] a decreasing sequence of numbers $\epsilon_j>0$ $(j\in\n)$,
\end{itemize}
such that the following conditions are satisfied for all $j\in\n$:
\begin{itemize}
\item[\rm (1$_j$)] $M_j$ is {\em homeomorphically isotopic} to $K_j$, meaning that 
\[
	(\imath_{M_j})_*(H_1(M_j;\z))=(\imath_{K_j})_*(H_1(K_j;\z))\subset H_1(M;\z)
\]
where $(\imath_{M_j})_*\colon H_1(M_j;\z)\to H_1(M;\z)$ and $(\imath_{K_j})_*\colon H_1(K_j;\z)\to H_1(M;\z)$ are the homomorphisms between the first homology groups with integer coefficients induced by the inclusion maps $\imath_{M_j}\colon M_j\to M$ and $\imath_{K_j}\colon K_j\to M$, respectively. (Notice that $(\imath_{M_j})_*$ and $(\imath_{K_j})_*$ are injective homomorphisms since $M_j$ and $K_j$ are $\Ocal(M)$-convex.)
\smallskip
\item[\rm (2$_j$)] $M_j\cap E=E_j$.
\smallskip
\item[\rm (3$_j$)] $|f_j(p)-f_{j-1}(p)|<\epsilon_{j-1}$ for all $p\in M_{j-1}$.
\smallskip
\item[\rm (4$_j$)] $f_j(p)=f(p)$ for all $p\in E_j$.
\smallskip
\item[\rm (5$_j$)] $f_j(M_j)\cap f(E\setminus E_j)=\varnothing$.
\smallskip
\item[\rm (6$_j$)] $f_j(M_j)\subset \mathring L_{j+1}$.
\smallskip
\item[\rm (7$_j$)] $f_j(bM_j)\cap L_j=\varnothing$.
\smallskip
\item[\rm (8$_j$)] $f_j(M_i\setminus\mathring M_{i-1})\cap L_{i-1}=\varnothing$ for all $i=1,\ldots,j$.
\smallskip
\item[\rm (9$_j$)] $0<\epsilon_j<\epsilon_{j-1}/2$.
\smallskip
\item[\rm (10$_j$)] If $g\colon M\to\c^2$ is a holomorphic map such that $|g(p)-f_j(p)|<2\epsilon_j$ for all $p\in M_j$, then $g|_{M_{j-1}}\colon M_{j-1}\to\c^2$ is an embedding with $g(M_{j-1})\subset\Omega$  and  $g(M_i\setminus\mathring M_{i-1})\cap L_{i-1}=\varnothing$ for all $i=1,\ldots,j-1$.
\end{itemize}
\end{lemma}
Condition {\rm (1$_j$)} in the lemma is equivalent to the existence of a compact domain $K_j'\subset M$ such that both $M_j$ and $K_j$ are strong deformation retracts of $K_j'$.

We defer the proof of Lemma \ref{lem:Induction} to the next subsections. Now, let us assume for a moment that the lemma holds true and let us show that it enables to complete the proof of Theorem \ref{th:main}. Set
\[
	D:=\bigcup_{j\in\z_+} M_j\subset M.
\]
Properties \eqref{eq:cupKj}, {\rm (A)}, {\rm (1$_j$)}, and {\rm (2$_j$)} guarantee that $D$ is a Runge domain in $M$ and satisfies condition {\rm (a)}; recall that $K=K_0=M_0$. On the other hand, by properties {\rm (3$_j$)} and {\rm (9$_j$)} in the lemma, there is a limit holomorphic map
\[
	\wt f:=\lim_{j\to\infty} f_j\colon D\to\c^2
\]
such that
\[
	|\wt f(p)-f_j(p)|<2\epsilon_j<\epsilon\quad \text{for all $j\in\z_+$}.
\]
Thus, taking into account properties \eqref{eq:cupLj}, \eqref{eq:cupEj}, {\rm (4$_j$)}, and {\rm (10$_j$)}, $j\in\n$, we infer that $\wt f$ is a proper holomorphic embedding from $D$ into $\Omega$ and meets conditions {\rm (b)}, {\rm (c)}, and {\rm (d)}; recall that $L_0=L$. 

Summarizing, the domain $D$ and the embedding $\wt f$ satisfy the conclusion of the theorem except for the final assertion that the domain $D$ can be chosen to be an open Riemann surface of hyperbolic type. In order to guarantee this condition it suffices to choose a Runge domain $M'\subset M$ of hyperbolic type such that $K\cup E\subset M$ and $M'$ is a deformation retract of $M$ (existence of such is well known, it can be easily proved, for instance, by a straightforward modification of the arguments in \cite{AlarconLopez2013JGA}), and apply the first part of Theorem \ref{th:main} (which we have just checked that holds true) to the same data but replacing the given open Riemann surface $M$ by $M'$. It follows that the domain $D\subset M'\subset M$ and the map $\wt f\colon D\to\Omega$ which we obtain in this way satisfy the first part of Theorem \ref{th:main} with respect to the open Riemann surface $M'$. Since $M'$ is of hyperbolic type and a deformation retract of $M$, we infer that $D$ is also of hyperbolic type and, by condition {\rm (a)}, a deformation retract of $M$. Moreover, since $M'$ is a Runge domain in $M$ and $D$ is a Runge domain in $M'$, $D$ is also a Runge domain in $M$. Therefore, the hyperbolic-type domain $D$ and the map $\wt f$ satisfy the conclusion of the theorem with respect to the open Riemann surface $M$ as well.

This completes the proof of Theorem \ref{th:main} granted Lemma \ref{lem:Induction}.

%
% The induction
%

\subsection{Proof of Lemma \ref{lem:Induction}}\label{sec:Induction}

We proceed by induction.
The basis is given by the already fixed $M_0$, $f_0$, and $\epsilon_0$; notice that, taking into account \eqref{eq:LhE} and that $f(K\cup E)\subset\Omega$, these objects satisfy conditions {\rm (1$_0$)}, {\rm (2$_0$)}, {\rm (4$_0$)}, {\rm (5$_0$)}, and {\rm (7$_0$)}, while the other ones are vacuous for $j=0$.

For the inductive step assume that  for some $j\in\n$ we already have sets $M_i$, maps $f_i$, and numbers $\epsilon_i$ satisfying the required properties for $i=0,\ldots,j-1$, and let us provide $M_j$, $f_j$, and $\epsilon_j$. We distinguish cases depending on the Euler characteristic of $K_j\setminus\mathring K_{j-1}$, which, by \eqref{eq:Euler}, is either $-1$ or $0$.

%
% Case 1
%

\subsection*{Case 1: Assume that the Euler characteristic $\chi(K_j\setminus\mathring K_{j-1})$ equals $-1$} 
In this case $K_j\setminus\mathring K_{j-1}$ is composed of finitely many compact annuli and exactly one {\em pair of pants}, i.e., a compact domain in $M$ which is homeomorphic to a topological sphere from which three open topological discs whose closures are pairwise disjoint have been removed. Thus, taking into account {\rm (1$_{j-1}$), there is a smooth Jordan arc $\gamma\subset M\setminus (E\cup \mathring M_{j-1})$, with the two endpoints in $bM_{j-1}$ and being otherwise disjoint from $M_{j-1}$, such that $M_{j-1}\cup\gamma$ is $\Ocal(M)$-convex and the image of the injective group homomorphism 
\[
	(\imath_{M_{j-1}\cup\gamma})_*\colon H_1(M_{j-1}\cup\gamma;\z)\to H_1(M;\z)
\]
equals $(\imath_{K_j})_*(H_1(K_j;\z))$, where $\imath_{M_{j-1}\cup\gamma}\colon M_{j-1}\cup\gamma\to M$ denotes the inclusion map. By properties {\rm (6$_{j-1}$)} and {\rm (7$_{j-1}$)} we may extend $f_{j-1}$, with the same name, to a smooth embedding $f_{j-1}\colon M_{j-1}\cup\gamma\to\c^2$ such that 
\begin{equation}\label{eq:fj-1gamma}
	f_{j-1}(\gamma)\subset \mathring L_j\setminus (L_{j-1}\cup f(E)).
\end{equation}
Thus, if we are given a number $\epsilon'>0$, then, by Mergelyan's theorem with interpolation (see e.g.\ \cite[Corollary 5.4.7]{Forstneric2017}), there are a small compact neighborhood $M_{j-1}'$ of $M_{j-1}\cup\gamma$ and a holomorphic embedding $f_{j-1}'\colon M_{j-1}'\to\Omega$ satisfying the following properties:
\begin{itemize}
\item[$\bullet$] $M_{j-1}'$ is homeomorphically isotopic to $K_j$ in the sense of {\rm (1$_j$)}.
\smallskip
\item[$\bullet$] $M_{j-1}'\cap E=E_{j-1}$.
\smallskip
\item[$\bullet$] $|f_{j-1}'(p)-f_{j-1}(p)|<\epsilon'$.
\smallskip
\item[$\bullet$] $f_{j-1}'(p)=f(p)$ for all $p\in E_{j-1}$. (Take into account {\rm (4$_{j-1}$)}.)
\smallskip
\item[$\bullet$] $f_{j-1}'(M_{j-1}')\cap f(E\setminus E_{j-1})=\varnothing$. (Take into account {\rm (5$_{j-1}$)}.)
\smallskip
\item[$\bullet$] $f_{j-1}'(M_{j-1}')\subset \mathring L_j$. (Take into account {\rm (6$_{j-1}$)}.)
\smallskip
\item[$\bullet$] $f_{j-1}'(M_{j-1}'\setminus \mathring M_{j-1})\cap L_{j-1}=\varnothing$. (Take into account {\rm (7$_{j-1}$)} and \eqref{eq:fj-1gamma}.)
\smallskip
\item [$\bullet$]$f_{j-1}'(M_i\setminus \mathring M_{i-1})\cap L_{i-1}=\varnothing$ for all $i=0,\ldots, j-1$. (See {\rm (8$_{j-1}$)}.)
\end{itemize}
In view of \eqref{eq:Euler}, this reduces the proof of the inductive step to the case when $\chi(K_j\setminus\mathring K_{j-1})=0$, which we now explain.

%
% Case 2
% 

\subsection*{Case 2: Assume that the Euler characteristic $\chi(K_j\setminus\mathring K_{j-1})$ equals $0$}
In this case $K_{j-1}$ is a strong deformation retract of $K_j$. 
We shall proceed in three steps, each one consisting of a different deformation procedure.

%
% Step 1
%

\medskip
\noindent{\em Step 1. Catching the points in $E_j\setminus E_{j-1}$.} The aim of this step is to approximate $f_{j-1}$ on $M_{j-1}$ by a holomorphic embedding which is defined on a compact domain in $M$ containing $M_{j-1}\cup E_j$ in its relative interior and which matches with $f$ everywhere on $E_j$ (see property {\rm (vi)} below). Our main tool in this step will be, again, the classical Mergelyan theorem with interpolation.

Assume that $E_j\setminus E_{j-1}\neq\varnothing$; otherwise we skip this step and proceed directly with Step 2 below. Since $E_j$ is finite then so is $E_j\setminus E_{j-1}$. For each point $p\in E_j\setminus E_{j-1}\subset M\setminus M_{j-1}$ (see {\rm (2$_{j-1}$)}), choose a smooth embedded Jordan arc $\gamma_p\subset M\setminus \mathring M_{j-1}$ having an endpoint in $bM_{j-1}\setminus E$ and meeting $bM_{j-1}$ transversely there, having $p$ as the other endpoint, and being otherwise disjoint from $M_{j-1}\cup E$. Choose the arcs $\gamma_p$, $p\in E_j\setminus E_{j-1}$, to be pairwise disjoint. Obviously, for each $p\in E_j\setminus E_{j-1}$ the arc $\gamma_p$ lies in the connected component of $M\setminus \mathring M_{j-1}$ containing the point $p$. Set 
\[	
	\Gamma:=\bigcup_{p\in E_j\setminus E_{j-1}}\gamma_p \subset M\setminus \mathring M_{j-1}
\]
and observe that
\begin{equation}\label{eq:Mj-1G}
	(M_{j-1}\cup\Gamma)\cap E=E_j
\end{equation}
(take into account {\rm (2$_{j-1}$)}).
Extend the holomorphic embedding $f_{j-1}$, with the same name, to a smooth embedding $f_{j-1}\colon M_{j-1}\cup\Gamma\to\Omega$ such that:
\begin{itemize}
\item[\rm (i)] $f_{j-1}(\Gamma)\subset \mathring L_j\setminus L_{j-1}$.
\smallskip
\item[\rm (ii)] $f_{j-1}(p)=f(p)$ for all $p\in E_j\setminus E_{j-1}$.
\end{itemize}
Existence of such extension is clear from {\rm (5$_{j-1}$)}, {\rm (6$_{j-1}$)}, {\rm (7$_{j-1}$)}, and the fact that $\mathring L_j\setminus L_{j-1}$ is a connected open set which contains $f(E_j\setminus E_{j-1})=\Lambda_j\setminus\Lambda_{j-1}$ (see \eqref{eq:cupLj} and \eqref{eq:Ej}). In view of {\rm (6$_{j-1}$)}, {\rm (7$_{j-1}$)}, and {\rm (i)} we have that
\begin{equation}\label{eq:fj-1G}
	f_{j-1}(M_{j-1}\cup\Gamma)\subset\mathring L_j
	\quad\text{and}\quad
	f_{j-1}(\Gamma\cup bM_{j-1})\cap L_{j-1}=\varnothing.
\end{equation}
Since $M_{j-1}\cup\Gamma$ is an $\Ocal(M)$-convex compact set, given a small number $\delta>0$ which will be specified later, Mergelyan's theorem with interpolation applied to $f_{j-1}\colon M_{j-1}\cup\Gamma\to \c^2$ furnishes a (connected) smoothly bounded $\Ocal(M)$-convex compact domain $R\subset M$ and a holomorphic embedding $\phi\colon R\to\c^2$ meeting the following requirements:
\begin{itemize}
\item[\rm (iii)] $R\cap E=E_j\subset M_{j-1}\cup\Gamma\subset\mathring R$ and $M_{j-1}$ is a strong deformation retract of $R$. (See \eqref{eq:Mj-1G}.)
\smallskip
\item[\rm (iv)] $|\phi(p)-f_{j-1}(p)|<\delta$ for all $p\in M_{j-1}\cup\Gamma$.
\smallskip
\item[\rm (v)] $\phi(R\setminus \mathring M_{j-1})\cap L_{j-1}=\varnothing$. (Take into account the second part of \eqref{eq:fj-1G}.)
\smallskip
\item[\rm (vi)] $\phi(p)=f(p)$ for all $p\in E_j$. (See {\rm (4$_{j-1}$)} and {\rm (ii)}.)
\end{itemize}
Moreover, in view of {\rm (iv)} and the first part of  \eqref{eq:fj-1G} and assuming that $\delta>0$ is chosen sufficiently small, we may and shall take $R$ close enough to $M_{j-1}\cup\Gamma$ so that
\begin{equation}\label{eq:phi(R)}
	\phi(R)\subset\mathring L_j.
\end{equation}

On the other hand, since $\phi(R\setminus\mathring M_{j-1})$ is compact and $L_{j-1}$ is compact and polynomially convex, in view of {\rm (v)} there is a polynomially convex compact set $\Lscr\subset\c^2$ such that
\begin{equation}\label{eq:Lscr}
	L_{j-1}\Subset \Lscr \Subset L_j
	\quad\text{and}\quad
	\phi(R\setminus\mathring M_{j-1})\cap \Lscr=\varnothing.
\end{equation}
Indeed, since $L_{j-1}$ is a polynomially convex compact set we have that for any neighborhood $U$ of $L_{j-1}$ there is another neighborhood $V=V(U)$ of $L_{j-1}$ such that if $Y$ is a compact subset of $V$, then the polynomial convex hull $\wh Y$ of $Y$ is contained in $U$. 
%(See e.\ g.\ \cite[p. 399]{Stout2007PM}.) 
Thus, if the neighborhood $U$ is chosen to lie in $\mathring L_j\setminus \phi(R\setminus\mathring M_{j-1})$, which is an open neighborhood of $L_{j-1}$ by \eqref{eq:cupLj} and {\rm (v)}, and the compact set $Y\subset V(U)$ is chosen with $L_{j-1}\subset\mathring Y$, then the polynomially convex compact set $\Lscr:=\wh Y$ meets the requirements in \eqref{eq:Lscr}.

This concludes the first deformation stage in the proof of the inductive step.

%
% Step 2
%

\medskip
\noindent{\em Step 2: Pushing the boundary out of $L_j$.} In this second step we shall deform $\phi(R)$ near its boundary in order to obtain an embedded complex curve whose boundary is disjoint from $L_j$ (see condition {\rm (x)} below). In order to do that we shall use the following approximation result by proper holomorphic embeddings into $\c^2$.

%%%
%%% Main Lemma
%%%
\begin{lemma}\label{lem:ML}
Let $L\subset\c^2$ be a polynomially convex compact set, let $R=\mathring R\cup bR$ be a compact bordered Riemann surface, let $K\subset \mathring R$ be a smoothly bounded compact domain, and assume that there is an embedding $\phi\colon R\to\c^2$ of class $\Ascr^1(R)$ such that
\begin{equation}\label{eq:RL}
	\phi(R\setminus\mathring K)\cap L=\varnothing.
\end{equation}
Then, for any $\epsilon>0$ there is a proper holomorphic embedding $\wt\phi \colon \mathring R\hra \c^2$ satisfying the following properties:
\begin{itemize}
\item[\rm (I)] $|\wt\phi(p)-\phi(p)|<\epsilon$ for all $p\in K$.
\smallskip
\item[\rm (II)] $\wt\phi(\mathring R\setminus \mathring K)\cap L=\varnothing$.
\end{itemize}
\end{lemma}

Lemma \ref{lem:ML} is an extension of Lemma 3.2 in Alarc\'on and L\'opez \cite{AlarconLopez2013JGA}, where the polynomially convex compact set is assumed to be a round ball. We shall prove the lemma by adapting the methods developed by Wold in \cite{Wold2006MZ} and by Forstneri\v c and Wold in \cite{ForstnericWold2009JMPA}, for embedding bordered Riemann surfaces in $\c^2$, in order to guarantee condition {\rm (II)}. 

In order to keep the story thread in the proof Lemma \ref{lem:Induction}, we defer the proof of Lemma \ref{lem:ML} to later on. Hence, assume that Lemma \ref{lem:ML} holds true and let us continue the proof of the inductive step.

Let $R_0\subset M$ be a smoothly bounded compact domain such that
\begin{equation}\label{eq:R0}
	M_{j-1}\cup\Gamma\Subset R_0\Subset R. 
\end{equation}
Note that $\phi(R\setminus\mathring R_0)\cap \Lscr=\varnothing$ by \eqref{eq:Lscr}, and hence Lemma \ref{lem:ML} may be applied to the polynomially convex compact set $\Lscr$, the compact bordered Riemann surface $R$, the domain $R_0$, and the embedding $\phi$. This furnishes, given a small number $\wt\delta>0$ which will be specified later, a proper holomorphic embedding $\wt\phi\colon\mathring R\hra\c^2$ such that:
\begin{itemize}
\item[\rm (vii)] $|\wt\phi(p)-\phi(p)|<\wt\delta$ for all $p\in R_0$.
\smallskip
\item[\rm (viii)] $\wt\phi(\mathring R\setminus \mathring R_0)\cap \Lscr=\varnothing$. 
\end{itemize}

Assuming that $\wt\delta>0$ is chosen sufficiently small, \eqref{eq:phi(R)} and {\rm (vii)} guarantee that 
\begin{equation}\label{eq:wtphiMj-1}
	\wt\phi(M_{j-1})\subset\wt\phi(R_0)\subset\mathring L_j.
\end{equation}
 On the other hand, since $R\subset M$ is a compact domain and $M_{j-1}$ is a strong deformation retract of $R$ (see {\rm (iii)}), $R\setminus \mathring M_{j-1}$ consists of finitely many, pairwise disjoint, smoothly bounded, compact annuli in $M$. By \eqref{eq:wtphiMj-1}, $\wt\phi$ maps the boundary components of these annuli which lie in $bM_{j-1}$ into $\mathring L_j$. Hence, since $\wt\phi\colon \mathring R\hra\c^2$ is a proper map, given a number
\begin{equation}\label{eq:tau}
	0<\tau<\frac12\dist(L_j,\c^2\setminus\mathring L_{j+1})
\end{equation} 
there is a (connected) smoothly bounded $\Ocal(M)$-convex compact domain $M_j$ with the following properties:
\begin{itemize}
\item[\rm (ix)] $R_0\Subset M_j\Subset R$ and $M_{j-1}$ is a strong deformation retract of $M_j$.
\smallskip
\item[\rm (x)] $\dist(\wt\phi(M_j), \c^2\setminus \mathring L_{j+1})>\tau$ and $\dist(\wt\phi(bM_j), L_j)>\tau$. In particular, we have $\wt\phi(M_j)\subset \mathring L_{j+1}$ and $\wt\phi(bM_j)\cap L_j=\varnothing$.
\smallskip
\item[\rm (xi)] $\wt\phi(M_j)\cap (\Lambda_{j+1}\setminus \Lambda_j)=\varnothing$.
\end{itemize}
Indeed, since the set $\Lambda_{j+1}\setminus\Lambda_j=f(E_{j+1}\setminus E_j)\subset \mathring L_{j+1}\setminus L_j$ is finite (see \eqref{eq:Ej}), condition {\rm (xi)} may be achieved by a slight deformation of $\wt\phi$ (for instance, by composing it with a small translation in $\c^2$). Alternatively, we may simply choose the domain $M_j\subset M$ such that $\wt\phi(M_j)$ is contained in a small neighborhood of $L_j$ in $\mathring L_{j+1}$ being disjoint from $\Lambda_{j+1}\setminus\Lambda_j$. (For the latter approach we have to choose $\tau>0$ in \eqref{eq:tau} sufficiently small to make possible the second inequality in {\rm (x)}, to be precise, we need $\tau<\dist(L_j,\Lambda_{j+1}\setminus\Lambda_j)$.)

This concludes the second deformation stage.

%
% Step 3
%

\medskip
\noindent{\em Step 3: Matching up with $f$ on $E_j$.} We shall now slightly perturb $\wt\phi$ to make it agree with $f$ everywhere on $E_j$.
%%%
%%% Lemma: Perturbing to preserve the hitting
%%%
In order to do that we shall use the following existence result for holomorphic automorphisms of $\c^2$.
\begin{lemma}\label{lem:hitting}
If $r>0$ is a number and $\Lambda\subset r\b=\{z\in\c^2\colon |z|< r\}$ is a finite set, then there are numbers $\eta>0$ and $\mu>0$ such that the following condition holds true. Given a number  $0<\beta<\eta$ and a map $\varphi\colon\Lambda\to\c^2$ such that
\begin{equation}\label{eq:lemvarphi}
	|\varphi(z)-z|<\beta\quad \text{for all $z\in\Lambda$},
\end{equation}
there is a holomorphic automorphism $\Psi\colon\c^2\to\c^2$ satisfying the following conditions:
\begin{itemize}
\item[\rm (I)] $\Psi(\varphi(z))=z$ for all $z\in\Lambda$.
\smallskip
\item[\rm (II)] $|\Psi(z)-z|<\mu\beta$ for all $z\in r\overline\b$.
\end{itemize}
\end{lemma}
Lemma \ref{lem:hitting} for $r=1$ is due to Globevnik (see \cite[Lemma 7.2]{Globevnik2016JMAA}); we shall prove the  general case as an application of this particular one.
(We point out that, alternatively, the proof given in \cite[Lemma 7.2]{Globevnik2016JMAA} may be easily adapted to work for arbitrary radious.)
Note that the number $\eta>0$ provided by the above lemma must be small enough so that every map $\varphi\colon\Lambda\to\c^2$ satisfying the inequality \eqref{eq:lemvarphi} for any $0<\beta<\eta$ is injective; otherwise condition {\rm (I)} would lead to a contradiction. 

We postpone the proof of Lemma \ref{lem:hitting} to later on. Hence, assume that Lemma \ref{lem:hitting} holds true and let us continue the discussion of the inductive step in the proof of Lemma \ref{lem:Induction}.

Since $L_{j+1}\subset\c^2$ is compact, there is a number $r>0$ such that 
\begin{equation}\label{eq:Lj+1}
	L_{j+1}\subset r\b=\{z\in\c^2\colon |z|<r\}.
\end{equation} 
Recall that the set $\Lambda_{j+1}=f(E_{j+1})\subset \mathring L_{j+1}$ given in \eqref{eq:Ej} is finite and let $\eta>0$ and $\mu>0$ be the numbers provided by Lemma \ref{lem:hitting} applied to $r$ and $\Lambda_{j+1}$ (take into account \eqref{eq:Lj+1}).
It is clear that
\begin{equation}\label{eq:etamu}
	\text{neither $\eta$ nor $\mu$ depend on the choice of the constants $\delta$ and $\wt\delta$}.\end{equation}

Also recall that $E_j\subset M_j$  (see {\rm (iii)}) and that $f|_{E_{j+1}}\colon E_{j+1}\to\Lambda_{j+1}$ is a bijection. Consider the map $\varphi\colon\Lambda_{j+1}\to\c^2$ given by
\[
	\Lambda_{j+1}\ni f(p)\longmapsto \varphi(f(p))=\left\{
	\begin{array}{ll}
	\wt\phi(p) & \text{if }p\in E_j, \bigskip
	\\
	f(p) & \text{if }p\in E_{j+1}\setminus E_j.
	\end{array}\right.
\]
In view of property {\rm (xi)} and the facts that  $\wt\phi\colon M_j\to\c^2$ is injective and that $f|_{E_j}\colon E_j\to f(E_j)=\Lambda_j\subset\Lambda_{j+1}$ is a bijection, we have that $\varphi$ is well defined and injective. Notice that $\varphi|_{\Lambda_{j+1}\setminus\Lambda_j}$ is the inclusion map. Moreover, conditions {\rm (vi)} and {\rm (vii)} ensure that 
\[
	|\varphi(z)-z|<\wt\delta\quad \text{for all $z\in\Lambda_{j+1}$};
\]
note that $E_j\subset M_j$ by {\rm (iii)}, \eqref{eq:R0}, and {\rm (ix)}.
Thus, assuming as we may that $\wt\delta>0$ has been chosen to be smaller than $\eta$ (see \eqref{eq:etamu}), Lemma \ref{lem:hitting} furnishes a holomorphic automorphism $\Psi\colon\c^2\to\c^2$ enjoying the following conditions:
\begin{itemize}
\item[\rm (xii)] $\Psi(\varphi(z))=z$ for all $z\in\Lambda_{j+1}$.
\smallskip
\item[\rm (xiii)] $|\Psi(z)-z|<\mu\wt\delta$ for all $z\in r\overline\b$.
\end{itemize}

This finishes the third (and final) deformation procedure in the proof of the inductive step.

We now prove the following.

\begin{claim}\label{cl:1-8}
If the numbers $\delta>0$ and $\wt\delta>0$ are chosen sufficiently small, then the smoothly bounded $\Ocal(M)$-convex compact domain $M_j$ and the holomorphic embedding
\[
	f_j:=\Psi\circ\wt\phi\colon M_j\to\c^2
\]
meet requirements {\rm (1$_j$)}--{\rm (8$_j$)} in the statement of the lemma. 
\end{claim}

Indeed, conditions {\rm (1$_j$)} and {\rm (2$_j$)} are ensured by {\rm (1$_{j-1}$)}, {\rm (iii)}, \eqref{eq:R0}, {\rm (ix)}, and the initial assumption that the Euler characteristic $\chi(K_j\setminus\mathring K_{j-1})$ equals $0$.  Notice now that {\rm (iv)}, {\rm (vii)}, {\rm (x)}, and \eqref{eq:Lj+1} give
\begin{equation}\label{eq:fj-fj-1}
	\left\{
	\begin{array}{ll}
	|f_j(p)-\wt\phi(p)|<\mu\wt\delta & \text{for all }p\in M_j,\bigskip
	\\
	|f_j(p)-\phi(p)|<(1+\mu)\wt\delta & \text{for all }p\in R_0,\bigskip
	\\
	|f_j(p)-f_{j-1}(p)|<\delta+(1+\mu)\wt\delta & \text{for all }p\in M_{j-1}\cup\Gamma.
	\end{array}\right.
\end{equation}
Thus, property {\rm (3$_j$)} is implied by \eqref{eq:fj-fj-1} provided that $\delta>0$ and $\wt\delta>0$ are chosen so that the inequality $\delta+(1+\mu)\wt\delta<\epsilon_{j-1}$ holds true; recall that the number $\mu>0$ does not depend on the choice of $\wt\delta$ (see \eqref{eq:etamu}). In order to check {\rm (4$_j$)} pick a point $p\in E_j$. We have that $f(p)\in\Lambda_j\subset\Lambda_{j+1}$, hence,
\[
	f_j(p)=\Psi(\wt\phi(p))\stackrel{\text{$p\in E_j$}}{=}\Psi(\varphi(f(p)))\stackrel{{\rm (xii)}}=f(p).
\]
Properties {\rm (6$_j$)} and {\rm (7$_j$)} follow from {\rm (x)} and \eqref{eq:fj-fj-1} provided that $\wt \delta$ is chosen with smaller than $\tau/\mu$, where $\tau>0$ is the number given in \eqref{eq:tau}; take into account that neither $\tau$ nor $\mu$ depend on the choice of $\wt \delta$. Now, {\rm (6$_j$)} and \eqref{eq:Ej} ensure that 
\begin{equation}\label{eq:fjMj}
	f_j(M_j)\cap f(E\setminus E_{j+1})=\varnothing.
\end{equation}
On the other hand, given a point $p\in E_{j+1}\setminus E_j$ we have that $\varphi(f(p))=f(p)$, hence,
\[
	\Psi(f(p))=\Psi(\varphi(f(p)))\stackrel{{\rm (xii)}}{=} f(p).
\]
Thus, since $\Psi\colon\c^2\to\c^2$ is a bijection and $f(p)\notin \wt\phi(M_j)$ by {\rm (xi)}, we infer that 
\[
	f(p)=\Psi(f(p))\notin\Psi(\wt\phi(M_j))=f_j(M_j);
\]
together with \eqref{eq:fjMj} we obtain {\rm (5$_j$)}. Finally, {\rm (8$_{j-1}$)} and \eqref{eq:fj-fj-1} ensure that
\begin{equation}\label{eq:Mi-Mi-1}
	f_j(M_i\setminus\mathring M_{i-1})\cap L_{i-1}=\varnothing\quad \text{for all }i=1,\ldots,j-1,
\end{equation}
provided that the numbers $\delta>0$ and $\wt\delta>0$ are chosen to satisfy 
\[
	\delta+(1+\mu)\wt\delta<
	\min\big\{\dist\big(f_{j-1}(M_i\setminus \mathring M_{i-1}), L_{i-1} \big)\colon i=1,\ldots,j-1\big\}.
\]
Note that the number in the right-hand side of the above inequality does not depend on the choices of $\delta$ and $\wt\delta$ and, in view of {\rm (8$_{j-1}$)}, is positive. On the other hand, \eqref{eq:Lscr} and \eqref{eq:fj-fj-1} give that 
\begin{equation}\label{eq:Mi-Mi-2}
	f_j(R_0\setminus \mathring M_{j-1})\cap L_{j-1}=\varnothing,
\end{equation}
provided that $(1+\mu)\wt\delta<\dist(L_{j-1},\c^2\setminus\mathring \Lscr)$; again, observe that the number in the right-hand side of this inequality is positive and does not depend on the choice of $\wt\delta$ (see \eqref{eq:Lscr}). Likewise, {\rm (viii)} and \eqref{eq:fj-fj-1} ensure that
\begin{equation}\label{eq:Mi-Mi-3}
	f_j(M_j\setminus \mathring R_0)\cap L_{j-1}=\varnothing,
\end{equation}
whenever that we choose $\wt\delta<\frac1{\mu} \dist(L_{j-1},\c^2\setminus\mathring \Lscr)$. Taking into account {\rm (ix)}, we infer from \eqref{eq:Mi-Mi-2} and \eqref{eq:Mi-Mi-3} that
\[
	f_j(M_j\setminus\mathring M_{j-1})\cap L_{j-1}=\varnothing.
\]
This and \eqref{eq:Mi-Mi-1} show condition {\rm (8$_j$)}. This concludes the proof of Claim \ref{cl:1-8}.

Once we have checked conditions {\rm (1$_j$)}--{\rm (8$_j$)}, to complete the proof of the inductive step it only remains to choose a number $\epsilon_j>0$ satisfying {\rm (9$_j$)} and {\rm (10$_j$)}. Taking into account {\rm (8$_j$)} and the facts that $M_{j-1}\subset\mathring M_j$ and that $f_j\colon M_j\to\Omega$ is a holomorphic embedding (see properties {\rm (A)} and {\rm (B)} in the statement of the lemma), such number exists by the Cauchy estimates, the compactness of $M_j$, and the openness of $\Omega$.

This concludes the proof of the inductive step in case $\chi(K_j\setminus\mathring K_{j-1})=0$, thereby proving the inductive step and completing the proof of Lemma \ref{lem:Induction} under the assumption that Lemma \ref{lem:ML} and Lemma \ref{lem:hitting} hold true.

%
% Completion of the proof of Theorem \ref{th:main}
%

\subsubsection*{Completion of the proof}\label{ss:lemmas}

In order to complete the proof of Lemma \ref{lem:Induction} we shall now prove Lemma \ref{lem:ML} and Lemma \ref{lem:hitting}. 
%
% Proof of Lemma \ref{lem:ML}
%

\begin{proof}[Proof of Lemma \ref{lem:ML}]
%We shall adapt the arguments in the proofs of Theorems 4.2 and 5.1 in \cite{ForstnericWold2009JMPA}; see also the proof of Lemma 3.2 in \cite{AlarconLopez2013JGA}. 
Let $L$, $R$, $K$,  $\phi$, and $\epsilon$ be as in the statement of Lemma \ref{lem:ML}. Assume without loss of generality that $R$ is a smoothly bounded compact domain in an open Riemann surface, $\wt R$, and, by Mergelyan's theorem and a shrinking of $\wt R$ around $R$ if necessary, that $\phi$ is a holomorphic embedding $\phi \colon \wt R\to \c^2$. Also, up to enlarging $K$ if necessary, we may and shall assume that $K$ is a strong deformation retract of $R$. 

Let $\pi_i\colon\c^2\to\c$ be the projection $\pi_i(\zeta_1,\zeta_2)=\zeta_i$, $i=1,2$. For each $z\in\c^2$, denote 
\[
	\Lambda_z:=\pi_1^{-1}(\pi_1(z))=\{(\pi_1(z),\zeta)\colon \zeta\in\c\}.
\]

Denote by $C_1,\ldots,C_m$ the connected components of $bR$. Since $L\subset\c^2$ is a polynomially convex compact set, we have that $\c^2\setminus L$ is a connected open set in $\c^2$, and hence path-connected. Thus, \eqref{eq:RL} enables us to choose pairwise disjoint smoothly embedded Jordan arcs $\lambda_1,\ldots,\lambda_m$ in $\c^2\setminus L$ meeting the following requirements for all $j\in\{1,\ldots,m\}$: 
\begin{itemize}
\item[\rm (a1)] $\lambda_j$ has an endpoint $w_j$ in $\phi(C_j)$ and is otherwise disjoint from $\phi(R)$.
\smallskip
\item [\rm (a2)]The other endpoint $z_j$ of $\lambda_j$ is an {\em exposed point} (with respect to the projection $\pi_1$) for the set $\phi(R)\cup(\bigcup_{k=1}^m\lambda_k)$ in the sense of  \cite[Def.\ 4.1]{ForstnericWold2009JMPA}.
\smallskip
\item[\rm (a3)] $\Lambda_{z_j}\cap L=\varnothing$. (Recall that $L$ is compact.)
\end{itemize}

Set $a_j:=\phi^{-1}(w_j)$ and note that $a_j$ is a well-defined point in $C_j$, $j=1,\ldots,m$. Let $V\subset R\setminus K$ be an open neighborhood of $\{a_1,\ldots,a_m\}$ in $R$ and $\epsilon_0>0$ be a number which will be specified later. Reasoning as in \cite[Proof of Theorem 4.2]{ForstnericWold2009JMPA} or \cite[Proof of Lemma 3.2]{AlarconLopez2013JGA}, we obtain, up to slightly deforming $\lambda_j$ near $w_j$ if necessary, a holomorphic embedding $\psi\colon R\to\c^2$ enjoying the following properties:
\begin{itemize}
\item[\rm (b1)] $|\psi(p)-\phi(p)|<\epsilon_0$ for all $p\in R\setminus V$.  
\smallskip
\item[\rm (b2)] $\dist(\psi(p),\lambda_j)<\epsilon_0$ for all $p\in V_j$, where $V_j$ is the component of $V$ containing $a_j$, $j=1,\ldots,m$.
\smallskip
\item[\rm (b3)] $\psi(a_j)=z_j$ is an exposed boundary point of $\psi(R)$, $j=1,\ldots,m$.
\end{itemize}
Moreover, choosing $\epsilon_0>0$ small enough, properties {\rm (b1)}, {\rm (b2)}, and \eqref{eq:RL} ensure that
\begin{itemize}
\item[\rm (b4)] $\psi(R\setminus \mathring K)\cap L=\varnothing$.
\end{itemize}

Now, given a number $\epsilon_1>0$ which will be specified later, arguing as in \cite[Proof of Theorem 5.1]{ForstnericWold2009JMPA} or \cite[Proof of Lemma 3.2]{AlarconLopez2013JGA}, there are numbers $\alpha_1,\ldots,\alpha_m\in\c\setminus\{0\}$ such that the rational shear map $g$ of $\c^2$ defined by
\[
	g(\zeta_1,\zeta_2)=\Big( \zeta_1 \,,\, \zeta_2+\sum_{j=1}^m \frac{\alpha_j}{\zeta_1-\pi_1(z_j)} \Big)
\]
satisifes the following conditions:
\begin{itemize}
\item[\rm (c1)] The projection $\pi_2$ maps the curve $\mu_j:=g(\psi(C_j\setminus\{a_j\}))\subset\c^2$ into an unbounded curve $\delta_j\subset\c$ and $\pi_2|_{\mu_j}\colon \mu_j\to\delta_j$ is a diffeomorphism near infinity, $j=1,\ldots,m$.
\smallskip
\item[\rm (c2)] The complement of the set $r\overline\d\cup (\bigcup_{j=1}^m \delta_j)\subset\c$ in $\c$ has no relatively compact connected components for any large enough $r>0$.
\smallskip
\item[\rm (c3)] $|g(z)-z|<\epsilon_1$ for all $z\in \psi(R\setminus V)$.
\smallskip
\item[\rm (c4)] $g(\psi(W\setminus \mathring K))\cap L=\varnothing$ where $W:=R\setminus \{a_1,\ldots,a_m\}$.
\end{itemize}
In order to ensure condition {\rm (c4)} we use {\rm (a3)} and {\rm (b4)}. Furthermore, setting
\[
	\wt\psi:=g\circ\psi|_W,
\]
it also holds that
\begin{itemize}
\item[\rm (c5)] there is a polynomially convex compact set $L_0\subset\wt\psi(W)$ in $\c^2$ such that $L\cup L_0$ is polynomially convex and $\wt\psi(K)\subset L_0$.
\end{itemize}

By the results in \cite{Wold2006MZ} (see also \cite[Proof of Theorem 5.1]{ForstnericWold2009JMPA}), given $\epsilon_2>0$ to be specified later, there are a Fatou-Bieberbach domain $D\subset\c^2$ and a biholomorphic map $\varphi\colon D\to\c^2$ (i.e., a Fatou-Bieberbach map) such that $\wt \psi(W)\cup L\subset D$, the boundaries $b(\wt \psi(W))\subset bD$, and
\begin{equation}\label{eq:epsilon2}
	|\varphi(z)-z|<\epsilon_2\quad \text{for all $z\in L\cup (D\cap L_0)$}.
\end{equation}

We claim that the map
\[
	\wt\phi:=\varphi\circ{\wt\psi}|_{\mathring R}\colon\mathring R\to\c^2
\]
is a proper holomorphic embedding $\mathring R\hra\c^2$ which satisfies the conclusion of the lemma. Indeed, provided that the positive numbers $\epsilon_0$, $\epsilon_1$, and $\epsilon_2$ are chosen sufficiently small, condition {\rm (I)} is ensured by {\rm (b1)}, {\rm (c3)}, {\rm (c4)}, and \eqref{eq:epsilon2}, while condition {\rm (II)} is implied by {\rm (c4)}, {\rm (c5)}, and \eqref{eq:epsilon2}. 
This concludes the proof.% of Lemma \ref{lem:ML}.
\end{proof}

%
% Proof of Lemma \ref{lem:hitting}
%

\begin{proof}[Proof of Lemma \ref{lem:hitting}]
Let $r$ and $\Lambda$ be as in the statement of Lemma \ref{lem:hitting}. Notice that if the lemma holds true for some $r>0$, then it also holds true for all numbers $r'\in(0,r)$; hence, we may and shall
assume without loss of generality that $r\ge 1$.

Since $\frac1{r}\Lambda$ is a finite set in $\b$, \cite[Lemma 7.2]{Globevnik2016JMAA} provides numbers $\eta>0$ and $\mu>0$ with the property that given $0<\delta<\eta$ and a map $\phi\colon\Lambda\to\c^2$ such that
\[
	|\phi(z)-z|<\delta\quad \text{for all $z\in\frac1{r}\Lambda$},
\]
there exists a holomorphic automorphism $\Phi\colon\c^2\to\c^2$ meeting
\begin{itemize}
\item[\rm (a)] $\Phi(\phi(z))=z$ for all $z\in\frac1{r}\Lambda$, and
\smallskip
\item[\rm (b)] $|\Phi(z)-z|<\mu\delta$ for all $z\in \overline\b$.
\end{itemize}

We claim that $\eta$ and $\mu$ satisfy the conclusion of Lemma \ref{lem:hitting}. Indeed, let $0<\beta<\eta$ and $\varphi\colon\Lambda\to\c^2$ be as in the statement of the lemma. In particular, the inequality \eqref{eq:lemvarphi} is satisfied. Consider the map $\phi\colon\frac1{r}\Lambda\to\c^2$ given by
\[
	\phi(z)=\frac1{r}\varphi(rz),\quad z\in\frac1{r}\Lambda.
\]
Since $r\ge 1$ and $0<\beta<\eta$, it turns out that for any $z\in \frac1{r}\Lambda$ we have
\[%\begin{equation}\label{eq:deltar}
	|\phi(z)-z| = \big| \frac1{r}\varphi(rz) -z\big| = \frac1{r}|\varphi(rz) -rz|\stackrel{\eqref{eq:lemvarphi}}{<}\frac{\beta}{r}<\eta,
\]%\end{equation}
and hence there is a holomorphic automorphism $\Phi\colon\c^2\to\c^2$ enjoying conditions {\rm (a)} and {\rm (b)} above with $\delta$ replaced by $\beta/r$. Consider the holomorphic automorphism $\Psi\colon\c^2\to\c^2$ given by
\[
	\Psi(z)=r\Phi\Big(\frac{z}{r}\Big),\quad z\in\c^2.
\]
Given $z\in\Lambda$ we have
\[
	\Psi(\varphi(z)) = r\Phi\Big(\frac{\varphi(z)}{r}\Big) = r \Phi\Big(\phi\Big(\frac{z}{r}\Big)\Big) 
	\stackrel{{\rm (a)}}{=} r\frac{z}{r}=z,
\]
which proves condition {\rm (I)} in the statement of Lemma \ref{lem:hitting}. On the other hand, for $z\in r\overline\b$ we infer that
\[
	|\Psi(z)-z| = \Big| r\Phi\Big(\frac{z}{r}\Big) -z\Big| = 
	r\Big| \Phi\Big(\frac{z}{r}\Big) - \frac{z}r\Big| \stackrel{{\rm (b)}}{<} r\mu\frac{\beta}{r}=\mu\beta.
\]
This shows condition {\rm (II)}, thereby concluding the proof of Lemma \ref{lem:hitting}.
\end{proof}

This completes the proof of Lemma \ref{lem:Induction}.
The proof of Theorem \ref{th:main} is complete. 
Theorem \ref{th:main-intro} is proved.

%%%%%%%%%%
%%%%%%%%%%
%%%%%%%%%%
%%%%%%%%%%   Section \ref{sec:CN}
%%%%%%%%%%
%%%%%%%%%%

%\section{Complex curves in pseudoconvex domains of $\c^N$ for $N\ge 3$}\label{sec:CN}

%%%%%%%%%%
%%%%%%%%%%
%%%%%%%%%%
%%%%%%%%%%   THANKS
%%%%%%%%%%
%%%%%%%%%%

\subsection*{Acknowledgements}
A.\ Alarc\'on is partially supported by the MINECO/FEDER grants no.\ MTM2014-52368-P and MTM2017-89677-P, Spain.

Part of the work on this paper was done while the author was visiting the Center for Advanced Study in Oslo (in December 2016) and the School of Mathematical Sciences at the University of Adelaide (in March 2017). He would like to thank these institutions for the hospitality and for providing excellent working conditions. 

The author wishes to thank Franc Forstneri\v c for helpful discussions.

%%%%%%%%%%
%%%%%%%%%%
%%%%%%%%%%
%%%%%%%%%%   THE BIBLIOGRAPHY
%%%%%%%%%%
%%%%%%%%%%

%{\bibliographystyle{abbrv} \bibliography{bibAFL}}

%%%%%%%%%%
%%%%%%%%%%
%%%%%%%%%%
%%%%%%%%%%   AFFILIATIONS
%%%%%%%%%%
%%%%%%%%%%

\vspace*{0.3cm}
\noindent Antonio Alarc\'{o}n

\noindent Departamento de Geometr\'{\i}a y Topolog\'{\i}a e Instituto de Matem\'aticas (IEMath-GR), Universidad de Granada, Campus de Fuentenueva s/n, E--18071 Granada, Spain.

\noindent  e-mail: {\tt alarcon@ugr.es}

\end{document}